\documentclass[12pt, reqno]{amsart}
\usepackage{amssymb,a4}
\usepackage{amsmath,amsthm,amscd}
\usepackage{amsthm}
\usepackage{amsfonts,amssymb}
\usepackage{mathrsfs}
\usepackage{amsmath}

\numberwithin{equation}{section}

\def\bC{\mathbb{C}}
\def\bN{\mathbb{N}}
\def\bZ{\mathbb{Z}}

\def\z{\mathbb Z}
\def\c{\mathbb C}
\def\n{\mathbb N}

\def\BE{\begin{equation}}
\def\EE{\end{equation}}

\newtheorem{theo}{{Theorem}}[section]
\newtheorem{lemm}[theo]{Lemma}
\newtheorem{rema}[theo]{Remark}
\newtheorem{defi}[theo]{Definition}

\newtheorem{prop}[theo]{Proposition}

\allowdisplaybreaks

\title[HC-module over the Ovsienko-Roger superalgebra]{Classification of simple Harish-Chandra modules over the Ovsienko-Roger superalgebra}

\author{ Munayim Dilxat}
\address{College of Mathematics and System Sciences, Xinjiang University, Urumqi 830046, Xinjiang, China}
\email{18396830969@163.com}

\author{Liangyun Chen}
\address{School of Mathematics ans Statistics, Northest Normal University, Changchun, 130024, China}
\email{chenly640@nenu.edu.cn}

\author{Dong Liu}
\address{College of Mathematics and System Sciences, Xinjiang University, Urumqi 830046, Xinjiang, China; Department of Mathematics, Huzhou University, Zhejiang Huzhou, 313000, China}
\email{liudong@zjhu.edu.cn}
\date{}

\begin{document}
\maketitle

\begin{abstract}  With the $\Omega$-operators for the Virasoro algebra \cite{BF} and the super Virasoro algebra in \cite{CL, CLL}, we get the $\Omega$-operators for the Ovsienko-Roger superalgebras in this paper and then use it to classify all simple cuspidal modules for the $\bZ$-graded and $\frac12\bZ$-graded Ovsienko-Roger superalgebras. By this result, we  can easily classify all simple Harish-Chandra modules over some related Lie superalgebras, including the $N=1$ BMS$_3$ algebra, the super $W(2,2)$, etc.



\end{abstract}

%\linenumbers

{\section{Introduction}

We denote by $\bZ, \bN, \bZ_+, \bC$ and $\bC^*$ the sets of all integers, non-negative integers, positive integers, complex numbers, and nonzero complex numbers, respectively. All vector spaces and algebras in this paper are over $\bC$. We denote by $U(L)$ the universal enveloping algebra of the Lie (super)algebra $L$ over $\bC$. Throughout this paper, by subalgebras, submodules for Lie superalgebras we mean subsuperalgebras and subsupermodules respectively.

Superconformal algebras may be viewed as
natural super-extensions of the Virasoro algebra and have been playing a fundamental role in  string theory and conformal field theory.  In \cite{Ka}, Kac classified all physical superconformal algebras:
namely, the $N = 0$ (the Virasoro algebra Vir), $N = 1$ (the super Virasoro algebras), $N = 2, 3$ and $4$ superconformal
algebras, the superalgebra $W(2)$ of all vector fields on the $N=2$ supercircle, and
a new superalgebra $CK(6)$. Representation theory for superconformal algebras has been the subject of intensive study.  It is a challenging problem to give complete classifications of simple weight modules with finite dimensional weight spaces for superconformal algebras.  
Based on the classification of simple jet modules introduced by Y. Billig in \cite{B} (see also \cite{Rao}), Billig and Futorny gave a complete classification of simple Harish-Chandra modules for Lie algebra of vector fields on a torus with the so-called $A$ cover theory in \cite{BF}. Recently, with the study of jet modules, such classifications were completed for many Lie superalgebras, the $N=1$ Ramond algebra in \cite{CL}, the Witt superalgebra in \cite{XL, BFI}, the affine-Virasoro superalgebra in \cite{CLW, HLW}, etc.  On the other hand, with classical methods for the classification of finite irreducible modules over the simple Lie algebras, \cite{LPX0} classified all strong Harish-Chandra modules over the $N=2$ superconformal algebras.   The above Lie superalgebras are all $\bZ$-graded. However, for the $\frac12\bZ$-graded Neveu-Schwarz algebras,  it meets many new difficulties. Recently \cite{CL} classified such modules for the $N=1$ Neveu-Schwarz algebra with some complicated constructions.

The Ovsienko-Roger Lie algebra $\widehat{\mathfrak L}_\lambda:={\rm Vir}\ltimes \mathcal F_{\lambda}$ was introduced in\cite{OR} to study the extensions of the Virasoro algebra by the densitiy module $\mathcal F_{\lambda}$. The algebras $\mathfrak L_{0}$ is  known in literature under the name the twisted Heisenberg-Virasoro algebra and
plays an important role in moduli spaces of curves \cite{ADKP}. The algebra $\mathfrak L_{-1}$ is  better known under the name the $W(2,2)$-algebra in the context of vertex operator algebras \cite{ZD} and  BMS/GCA correspondence \cite{Ba0}. Moreover with the study of the Ovsienko-Roger Lie algebra $\mathfrak L_\lambda$,  Harish-Chandra modules for many Lie algebras related to the Virasoro algebra can be classified (see \cite{L, LPX2}). 

Motivited it we introduce the Ovsienko-Roger Lie superalgebra $\widehat{\mathfrak L}(\lambda, \epsilon)={\rm Vir}\ltimes \mathcal F_{\lambda}, \epsilon=0, \frac12$, where $\mathcal F_{\lambda}=\sum_{i\in\mathbb Z+\epsilon}G_i$ is the odd part of $\widehat{\mathfrak L}(\lambda, \epsilon)$. $\widehat{\mathfrak L}(0, \epsilon)$ is named the Kuper algebras, which is connected with the super Camassa-Holm-type systems (see \cite{G}).
It is well known that $\Omega$-operators plays an important role in the classification of irreducible cuspidal modules  over many Lie (super)algebras. With the $\Omega$-operators for the Virasoro algebra \cite{BF} and the super Virasoro algebra in \cite{CL, CLL}, we get the $\Omega$-operators on the cuspidal modules for the Ovsienko-Roger superalgebras (Lemma \ref{Omegaoper-OS} below) and then give a uniform method to classify all simple Harish-Chandra modules for the $\bZ$-graded and $\frac12\bZ$-graded Ovsienko-Roger superalgebras (Theorem \ref{cuspidal} below). With this result we can easily classify all simple Harish-Chandra modules for many related Lie superalgebras (see Section 5).  We shall note that we just do our research for $\lambda=-\frac12$ in the whole paper although our calculations and proofs are all suitable for any $\lambda\in\mathbb C$.

The paper is organized as follows. In Section \ref{pre}, we collect some basic results for our study. Simple cuspidal modules are classified in Section 3.  In Section 4, we classify all simple Harish-Chandra modules for the Ovsienko-Roger superalgebra. Finally, with this classification we  can classify all simple Harish-Chandra modules over some related Lie superalgebras including the $N=1$ BMS$_3$ algebra, the super $W(2,2)$, etc. in Section \ref{final}.

\section{Preliminaries}
\label{pre}

In this section, we collect some basic definitions and results for our study.

By definition, as a vector space over $\bC$, the Virasoro algebra ${\rm Vir}$ has a basis $\{L_m, C
\mid m\in\bZ\}$, subject to the following relations:
\begin{eqnarray}\label{def1}
&&[L_m, L_n]=(n-m)L_{m+n}+\delta_{m+n, 0}{1\over 12}(m^3-m)C, \ \forall m, n\in\bZ.
\end{eqnarray}

It is well known that the intermediate series ${\rm Vir}$-modules is
$${\mathcal A}_{a,\; b}:=\sum_{i\in\bZ}\bC v_i \ {\rm with} \ L_mv_i=(a+i+bm)v_{m+i}, Cv_i=0, \forall m, i\in\bZ.$$
${\mathcal A}_{a,\; b}$ is simple if and only if $a\not\in\bZ$ or $b\ne 0, 1$. As usual, we use $\mathcal A_{a,\; b}'$ to denote by the irreducible sub-quotient module of ${\mathcal A}_{a,\; b}$ (see \cite{KS}).

If $a\in\bZ$, then ${\mathcal A}_{a,\; b}\cong {\mathcal A}_{0,\; b}$. So we always suppose that $a\notin\bZ$ or $a=0$ in ${\mathcal A}_{a,\; b}$.
As in \cite{Fuk, OR}, we denote by ${\mathcal F}_\lambda={\mathcal A}_{0,\; \lambda},  \lambda\in\bC$, which is also called density module of the Virasoro algebra.

Motivated \cite{OR}, for $\epsilon=0, \frac12$, we can define the Ovsienko-Roger superalgebra $\widehat{\mathfrak L}(\epsilon):={\rm Vir}\ltimes \mathcal F_{-\frac12}$.  More precisely $\widehat{\mathfrak L}(\epsilon)$ has a  basis $\{L_n, G_r,C\,|\, n\in\bZ, r\in\bZ+\epsilon\}$ with the brackets
\begin{align*}
[L_m,L_n]&=(n-m)L_{m+n}+\delta_{m+n,0}\frac{1}{12}(n^3-n)C,\\
[L_m, G_r]&=(r-\frac12m)G_{r+m},\\
[G_r,G_s]&=0, \ \forall m, n\in\bZ, r, s\in\bZ+\epsilon.
\end{align*}
Here we shall note that the odd part of $\widehat{\mathfrak L}(\epsilon)$ is spanned by $\{G_{n}\mid n\in\bZ+\epsilon\}$.

In the case of $\epsilon=0$, $\widehat{\mathfrak L}(0)$ can be realized as the affine-Virasoro superalgebra $\bC x\otimes \bC[t, t^{-1}]\rtimes {\rm Vir}$, where $\bC x$ is the one-dimensional abelian Lie superalgebra. All simple Harish-Chandra modules over $\widehat{\mathfrak L}(0)$ were classified in \cite{CLW}.

For the case of $\epsilon=\frac12$, it encounters many new difficulties to classify such modules with usual methods. This paper give a uniform new method to consider both cases of $\epsilon=0$ and $\epsilon=\frac12$.  For convenience,  we just write our research for the case of $\epsilon=\frac12$.
So from now denote by $\widehat{\mathfrak L}=\widehat{\mathfrak L}(\frac12)$ in brief, and $\mathfrak L$ the quotient algebra $\widehat{\mathfrak L}/\bC C$.
Clearly, $\mathfrak L$ is a $\frac{1}{2}\bZ$-graded Lie superalgebra with ${\mathfrak L}_i=\bC L_i, \forall i\in\bZ$ and ${\mathfrak L}_{i+\frac12}=\bC G_{i+\frac12}, \forall i\in\bZ$.

The subalgebra of $\mathfrak L$ spanned by $\{L_k\,|\, k\in\bZ\}$ is isomorphic to the Witt algebra $W$. The following results for $W$-modules will be used.
\begin{lemm}\label{Omegaoper}$(${\cite[Corollary 3.7]{BF}}$)$
Let $\Omega_{k, s}^{(m)}=\sum\limits_{i=0}^m(-1)^i\binom{m}{i}L_{k-i}L_{s+i}$.
For every $\ell\in\bZ_+$ there exists $m\in\bZ_+$ such that for all $k, s\in\bZ$, $\Omega_{k,s}^{(m)}$ annihilate every cuspidal $W$-module with a composition series of length $\ell$.
\end{lemm}

All simple $W$-modules with finite dimensional weight spaces  were  classified in \cite{Ma}.

\begin{theo}\label{thm-vir}\cite{Ma}
Let $V$ be a simple  $W$-module with finite dimensional weight spaces.
Then $V$ is a highest weight module, lowest weight module, or $\mathcal A_{a, b}'$ for some $a, b\in\bC$.
\end{theo}

\section{Simple cuspidal $\mathfrak L$-module}

In this section, we shall consider cuspidal $\mathfrak L$-modules. The following result was given in \cite{CL} (also see \cite{CLL}).
\begin{lemm}\label{Omegaoper-S}\cite{CL}
Let $V$ be a cuspidal
${\mathfrak L}$-module. Then there exists $m\in\bZ_+$ such that for all $r\in\bZ+\frac12, s\in\bZ$, $\overline{\Omega}_{r, s}^{(m)}$ annihilate $V$, where
$\overline{\Omega}_{r, s}^{(m)}=\sum\limits_{i=0}^m(-1)^i\binom{m}{i}G_{r-i}L_{s+i}$.
\end{lemm}

\begin{lemm}\label{Omegaoper-OS} Let $V$ be a cuspidal
${\mathfrak L}$-module. Then there exists $m\in\bZ_+$ such that
$\underline{\Omega}_{r, s}^{(m)}$ annihilate $V$, where
\begin{equation}
\underline{\Omega}_{r, s}^{(m)}:=\sum\limits_{i=0}^m(-1)^i\binom{m}{i}G_{r-i}G_{s+i}, \ \forall r, s\in\bZ+\frac12. \label{omega3}
\end{equation}
\end{lemm}

\begin{proof}

For the cuspidal module $V$, by Lemma \ref{Omegaoper-S}, there exists $m\in\bZ_+$ such that for all $r\in\bZ+\frac12, s\in\bZ$, $\overline{\Omega}_{r, s}^{(m)}V=0$, it is
\begin{equation}
\sum\limits_{i=0}^m(-1)^i\binom{m}{i}G_{r-i}L_{s+i}V=0, \ \forall r\in\bZ+\frac12, s\in\bZ. \label{omega1}
\end{equation}

By action of $G_t$ on \eqref{omega1}, we get
\begin{equation}
\frac12\sum\limits_{i=0}^m(-1)^i\binom{m}{i}(s+i-2t)G_{r-i}G_{s+t+i}V=0, \ \forall r\in\bZ+\frac12, s\in\bZ. \label{omega2}
\end{equation}

Choosing $t=t_1, t_2, t_1\ne t_2$ in \eqref{omega2}, we get the lemma.
\end{proof}

\begin{lemm}\label{nilplemma} Let $V$ be a simple cuspidal
${\mathfrak L}$-module. Then there exists $N\in\bZ_+$ such that $\mathfrak L_{\bar1}^NV=0$.
\end{lemm}
\begin{proof}
By Lemma \ref{Omegaoper-OS}, we can get \eqref{omega3}. Multiply  both sides of \eqref{omega3} by  $G_{s+1} G_{s+2}\cdots G_{s+m}$,  
$G_{r-j+1}\cdots G_{r-1}G_rG_{s+j+1}\cdots G_{s+m}, 1\le j\le m$ to get

\begin{align}
&G_{r}G_{s}G_{s+1}\cdots G_{s+m}V=0,\label {nilp-1} \\
&G_{r-j}\cdots G_{r-1}G_rG_{s+j}G_{s+j+1}\cdots G_{s+m}V=0,\ \forall 1\le j\le m. \label {nilp-j}
\end{align}

Fix some $s\in\bZ+\frac12$ and set ${\mathcal O}_n=\{s, s+1, s+2, \cdots, s+n\}$. By \eqref{nilp-1} the following identity
\begin{equation}
G_{r_0}G_{r_1}\cdots G_{r_{m+1}}V=0 \label{nilp}
\end{equation}
holds for all  $r_0, r_1, \cdots, r_{m+1}\in {\mathcal O}_{m+1}$.

By \eqref{nilp-1} and \eqref{nilp-j} we see that $\eqref{nilp}$ holds  for all  $r_0, r_1, \cdots, r_{m+1}\in {\mathcal O}_{m+2}$.

We shall use the induction on $k$ to prove that $$G_{r_0}G_{r_1}\cdots G_{r_{m+1}}V=0$$ for all  $r_0, r_1,\cdots, r_{m+1}\in {\mathcal O}_{m+k}$ for all $k\ge 1$.
Then, according to the arbitariness of $s$, we get the lemma by choosing $N=m+2$.

Suppose that \eqref{nilp} holds for all  $r_0, r_1, \cdots, r_{m+1}\in {\mathcal O}_{n}$ and some $n>m+1$.  Now we shall prove that
\begin{equation}
G_{r_0}G_{r_1}\cdots G_{r_m}G_{s+n+1}V=0 \label{nilp-n+1}
\end{equation}
holds for all  $r_0<r_1<\cdots<r_{m}\in {\mathcal O}_{n}$.

\noindent{\bf Case 1.}  $r_0=s+n-m$.

In this case $r_i=s+n-m+i$ for any $i=1, 2, \ldots, m$.  So \eqref{nilp-n+1} follows from \eqref{nilp-1} directly.

\noindent{\bf Case 2.}  $r_0=s+n-m-k$ for some $1\le k\le n-m$.

Replaced by $s, r$ by $s+n-m+1, s+n-k$ in \eqref{omega3}, respectively, we get
\begin{align*}
&\big(G_{s+n-k}G_{s+n-m+1}-\binom{m}{1}G_{s+n-k-1}G_{s+n-m+2}\\
&+\cdots+(-1)^{m-1}\binom{m}{m-1}G_{s+n-k-m+1}G_{s+n}+(-1)^mG_{r_0}G_{s+n+1}\big)V=0.
\end{align*}
So we get
\begin{equation}
G_{r_0}G_{s+n+1}V\subset (\sum_{r_i, r_j\in\mathcal O_n} G_{r_i}G_{r_j})V.
\end{equation}
In this case  \eqref{nilp-n+1} follows by inductive hypothesis.
\end{proof}

\begin{theo}\label{cuspidal}
Let $V$ be a simple cuspidal
${\mathfrak L}$-module. Then $V$ is isomorphic to the Harish-Chandra module of the intermediate series: $V=\sum v_i\cong {\mathcal A}_{a, b}'$ for some $a, b\in\c$ with
$L_mv_i=(a+i+bm)v_{m+i},  G_rv_i=0$
for all $m,i\in\z, r\in\bZ+\frac12$.
\end{theo}

\begin{proof}
Clearly $\dim\,V_i\le p$ for some positive integer $p$ holds for almost $i\in\z$ and $C$ acts on $V$ as zero (see \cite{KS}). Now ${\mathfrak L}_{\bar1}^iV$ is ${\mathfrak L}$-submodule since ${\mathfrak L}_{\bar1}^{i+1}V\subset {\mathfrak L}_{\bar1}^iV$ for all $i\in\n$. So ${\mathfrak L}_{\bar1}V=V$ or ${\mathfrak L}_{\bar1}V=0$.

By Lemma \ref{nilplemma}, we get
\begin{equation}{\mathfrak L}_{\bar1}^NV=0. \label{grg01}\end{equation}
If ${\mathfrak L}_{\bar1}V=V$ then ${\mathfrak L}_{\bar1}^NV=V=0$, which is a contradiction.
So ${\mathfrak L}_{\bar1}V=0$ and the proposition follows from  Theorem \ref{thm-vir}.
\end{proof}

\section{Simple Harish-Chandra module}

Now we can classify all simple Harish-Chandra modules over $\widehat{\mathfrak L}$.
The following result is well-known.
\begin{lemm}\label{weightupper}
Let $M$ be a weight module with finite dimensional weight spaces for the Virasoro algebra with $\mathrm{supp}(M)\subseteq\lambda+\bZ$. If for any $v\in M$, there exists $N(v)\in\bN$ such that $L_iv=0, \forall i\geq N(v)$, then $\mathrm{supp}(M)$ is upper bounded.
\end{lemm}

\begin{lemm}\label{appN=1}
Suppose $M$ is a simple weight $\widehat{\mathfrak L}$-module with finite dimensional weight spaces which is not cuspidal, then $M$ is a highest (or lowest) weight module.
\end{lemm}
\begin{proof} It is essentially same as that of Lemma 4.2 (1) in \cite{CL}.

Fix a $\lambda\in\mathrm{supp}(M)$. Since $M$ is not cuspidal, then there is a $k\in\frac12\bZ$ such that $\dim M_{-k+\lambda}>2(\dim M_\lambda+M_{\lambda+\frac12}+\dim M_{\lambda+1})$. Without loss of generality, we may assume that $k\in\bN$. Then there exists a nonzero element $w\in M_{-k+\lambda}$ such that $L_kw=L_{k+1}w=G_{k+\frac12}w=0$. Therefore, $L_iw=G_{i-\frac12}w=0$ for all $i\geq k^2$, since $[\widehat{\mathfrak L}_i,\widehat{\mathfrak L}_j]=\widehat{\mathfrak L}_{i+j}$.

It is easy to see that $M'=\{v\in M\,|\,\dim\widehat{\mathfrak L}^+v<\infty\}$ is a nonzero submodule of $M$, here $\widehat{\mathfrak L}^+=\sum\limits_{n\in\bZ_+}(\bC L_n+\bC G_{n-\frac12})$. Hence $M=M'$. So, Lemma \ref{weightupper} tells us that $\mathrm{supp}(M)$ is upper bounded, that is $M$ is a highest weight module.
\end{proof}

Combining Lemma \ref{appN=1} and Theorem \ref{cuspidal}, we can get the following result.
\begin{theo}\label{main1}
 Any simple $\widehat{\mathfrak L}$ module with finite dimensional weight spaces is a highest weight module, lowest weight module, or is isomorphic to  $\mathcal A_{a,b}'$ for some $a, b\in\bC$.
\end{theo}

\section{Applications}\label{final}

Some Lie superalgebras were constructed in \cite{WCB} as an application of the classification of Balinsky-Novikov super-algebras with dimension $2|2$. As applications of the above results, we can classify all Harish-Chandra modules over many Lie superalgebras listed in Table 7 in \cite{WCB}.

\subsection{The Lie superalgebra $\frak q$}
By definition the Lie superalgebra $\frak q=\frak q_{\bar0}+\frak q_{\bar1}$, where $\frak q_{\bar0}:=\bC\{L_m, H_m,  C\mid m\in\bZ\}$ and $\frak q_{\bar1}=\bC\{G_p\mid p\in\bZ+\frac12\}$, is a subalgebra of the $N=2$ Neveu-Schwarz superconformal algebra, with the following relations:
\begin{eqnarray}
&& [L_m,L_n]=(n-m)L_{n+m}+{1\over12}(n^3-n)C,\nonumber\\
&&[H_m,H_n]={1\over3}m\delta_{m+n,0}C,\ \ \ \ \ \ \ \ \ \  [L_m, H_n]=nH_{m+n},\nonumber\\
&&\label{gd2}
[L_m,G_p]=(p-\frac{m}{2})G_{p+m},\ \ \ \ \ \ \ \ [H_m,G_p]=G_{m+p},\label{qdef3}\\
&&[G_p, G_q]=0,\nonumber
\end{eqnarray}
for $m,n\in\bZ,\,p, q\in\bZ+\frac12$.

Clearly $\mathcal A_{a, b, c}=\sum_{i\in\bZ}\bC v_i$ is a $\frak q$-module with
\begin{eqnarray*}
&& L_mv_i=(a+bm+i)v_{m+i}, \ H_mv_i=cv_{m+i}, \ G_rv_i=0, \forall m, i\in\bZ, r\in\bZ+\frac12.
\end{eqnarray*}
Moreover ${\mathcal {A}}_{a, b, c}$ is simple if and only if $a\not\in\bZ$, or $b\ne 0, 1$ or $c\ne 0$. We also use $\mathcal A_{a, b, c}'$ to denote by the simple sub-quotient of $\mathcal A_{a, b, c}$.

\begin{prop}\label{cus-q}
Any simple cuspidal $\frak q$-module $V$ is isomorphic to the module $\mathcal A_{a, b, c}'$ of the intermediate series  for some $a, b, c\in\bC$.
\end{prop}
\begin{proof}
Clearly, the subalgebra $\frak q':={\rm span}\{L_m, G_r, C\mid m\in\bZ, r\in\bZ+\frac12\}$ is isomorphic to $\mathfrak L$.
By Theorem \ref{cuspidal}, we can choose an irreducible $\frak q'$-module $V'$ with $G_rV'=0$ for all $r\in\bZ+\frac12$. In this case we have $V={\rm Ind}_{\frak q'}^{\frak q}V'$.
Moreover we have $G_rV=0$ for all $r\in\bZ+\frac12$ by  \eqref{qdef3}.
In this case the $\frak q$-module $V$ is simple if and only if $V$ is a simple $\frak q_{\bar0}$-module.
So the proposition follows from the main theorem in \cite{LvZ}.
\end{proof}

\begin{rema}
Proposition \ref{cus-q} palys a key role in the classification of all simple cuspidal weight module for the $N=2$ Neveu-Schwarz superconformal algebra, see \cite{LPX0}.

\end{rema}

\subsection{The $N=1$ BMS$_3$ algebra}

The Bondi-Metzner-Sachs (BMS$_3$) algebra  is the symmetry algebra
of asymptotically flat three-dimensional spacetimes \cite{BBM}. It is the semi-direct product of the
Virasoro algebra with its adjoint module. The $N=1$ super-BMS$_3$ is a minimal supersymmetric extension of the BMS$_3$ algebra, which has been introduced  to describe the asymptotic structure of the $N=1$ supergravity in \cite{BDMT}.
\begin{defi}\label{BMS} The $N$=1 BMS$_3$ superalgebra $\mathcal B$ is a Lie superalgebra with a basis $\{L_m,  I_m,  Q_r,  C_1,  C_2\mid m\in\bZ,  r\in\bZ+\frac12\}$, with the following commutation relations:
\begin{eqnarray*}\label{brackets}
{[L_m, L_n]}&=&(m-n)L_{m+n}+{1\over12}\delta_{m+n, 0}(m^3-m)C_1,\\
{[L_m, I_n]}&=&(m-n)I_{m+n}+{1\over12}\delta_{m+n, 0}(m^3-m)C_2,\\
{[Q_r, Q_s]}&=&2I_{r+s}+{1\over3}\delta_{r+s, 0}\left(r^2-\frac14\right)C_2,\\
{[L_m, Q_r]}&=&\left(\frac{m}{2}-r\right)Q_{m+r},\\
{[I_m,I_n]}&=&[M_n,Q_r]=0, \quad [C_1,\frak g]=[C_2, \frak g]=0
\end{eqnarray*} for any $m, n\in\bZ, r, s\in\bZ+\frac12$.
\end{defi}

Note that $\mathcal B=\mathcal B_{\bar0}+\mathcal B_{\bar1}$, where $\mathcal B_{\bar0}:=\bC\{L_m, I_m,  C_1, C_2\mid m\in\bZ\}$ and $\mathcal B_{\bar1}=\bC\{Q_p\mid p\in\bZ+\frac12\}$.  The quotient algebra $\mathcal B/J$ is isomorphic to $\mathfrak L$, where $J=\bC\{I_m\mid m\in\bZ\}$.

Clearly the Vir-module $\mathcal A_{a, b}$ can be become a $\mathcal B$-module with the trivial actions of $I_m, Q_r$ for any $m\in\bZ, r\in\bZ+\frac12$.

\begin{prop}\label{cus-BMS}
Any simple cuspidal $\mathcal B$-module $V$ is isomorphic to the module $\mathcal A_{a, b}'$ of the intermediate series  for some $a, b\in\bC$.
\end{prop}
\begin{proof}
Clearly, the subalgebra $\mathcal B_{\bar0}$ is isomorphic to $W(2, 2)$.
By Theorem 4.6 in \cite{GLZ}, we can choose an irreducible $\mathcal B_{\bar0}$-module $V'$ with $I_mV'=0$ for all $m\in\bZ$. In this case we have $V={\rm Ind}_{\mathcal B_{\bar0}}^{\mathcal B}V'$.
Moreover we have $I_mV=0$ and $[G_r, G_s]V=0$ for all $m\in\bZ, r, s\in\bZ+\frac12$ by  Definition \ref{BMS}.
In this case the $\mathcal B$-module $V$ is simple if and only if $V$ is a simple $\mathcal B/J$-module.
So the proposition follows from Theorem \ref{cuspidal}.
\end{proof}

\subsection{The super $W(2,2)$ algebra}

By definition, the super $W(2,2)$ algebra is the Lie superalgebra $SW(2,2):= \bC\{L_m, I_m,  G_r, Q_r,  C_1,  C_2\mid m\in\bZ,  r\in\bZ+\frac12\}$, with the following relations:
\begin{eqnarray}\label{brackets}
\begin{array}{lllll}
&[L_m, L_n]=(m-n)L_{m+n}+{1\over12}\delta_{m+n, 0}(m^3-m)C_1,\\
&[L_m, I_n]=(n-m)I_{m+n}+{1\over12}\delta_{m+n, 0}(m^3-m)C_2,\\
&[G_r, G_s]=2L_{r+s}+{1\over3}\delta_{r+s, 0}(r^2-\frac14)C_1,\\
&[G_r, Q_s]=2I_{r+s}+{1\over3}\delta_{r+s, 0}(r^2-\frac14)C_2,\\
&[L_m, G_r]=(\frac{m}{2}-r)G_{m+r}, \   [L_m, Q_r]=(\frac{m}{2}-r)Q_{m+r}, \\
&[I_m,G_r]=(\frac{m}{2}-r)Q_{m+r},
\end{array}
\end{eqnarray} for any $m, n\in\bZ, r, s\in\bZ+\frac12$.

Note that $SW(2,2)=SW(2,2)_{\bar0}+SW(2,2)_{\bar1}$, where $SW(2,2)_{\bar0}:=\bC\{L_m, I_m,  C_1, C_2\mid m\in\bZ\}$ and $SW(2,2)_{\bar1}=\bC\{G_p, Q_p\mid p\in\bZ+\frac12\}$.

Clearly the subalgebra generated by $\{L_m, G_r~|~ m\in\bZ, r\in\bZ+\frac12\}$ is isomorphic to the $N=1$ Neveu-Schwarz algebra $\frak S$.
From \cite{S2} we see that $S(a, b)$ or $\Pi S(a, b)$ is the Harich-Chandra module of intermediate series over $\frak S$ for some $a, b\in\bC$, where $S_{a, b}$ defined as follows:
\begin{eqnarray*}
&&S_{a, b}:=\sum_{i\in\bZ}\bC x_i+\sum_{k\in\bZ+\frac12}\bC y_k\ \hbox{with}\\
L_nx_i&=&(a+bn+i)x_{i+n}, \ L_ny_k=(a+(b+\frac12)n+k)y_{k+n},\\
G_rx_i&=&(a+i+2rb)y_{r+i},  \hskip60pt   G_ry_k=-x_{r+k},
\end{eqnarray*} for all $n, i\in\bZ, r, k\in\bZ+\frac12$.

Moreover $S_{a, b}$ is simple if and only if $a\not\in\bZ$ or $a\in\bZ$ and $b\ne0, \frac12$. We also use $S_{a, b}'$ to denote by the simple sub-quotient of $S_{a, b}$.

Clearly the $\frak  S$-modules $S_{a, b}$ and $\Pi S_{a, b}$ become  $SW(2,2)$-modules with trivial actions of $I_m, Q_{m+\frac12}$ for any $m\in\bZ$.

\begin{prop}\label{cus-p}
Any simple cuspidal $SW(2, 2)$-module $V$ is isomorphic to $S_{a, b}'$ or $\Pi S_{a, b}'$  for some $a, b\in\bC$.
\end{prop}
\begin{proof}
Set $\frak p={\rm span}\{L_m, I_m, Q_r, C_1, C_2\mid m\in\bZ, r\in\bZ+\frac12\}$. By Proposition \ref{cus-BMS}  we can choose a simple $\frak p$-module $V'$ with $I_mV'=Q_rV'=0$ for all $m\in\bZ, r\in\bZ+\frac12$. In this case we have $V={\rm Ind}_{\frak p}^{SW(2,2)}V'$.
Moreover we have $I_mV=Q_rV=0$ for all $m\in\bZ, r\in\bZ+\frac12$ by  Definition \ref{brackets}. In this case the $SW(2, 2)$-module $V$ is simple if and only if $V$ is a simple $\frak S$-module.
So the proposition follows from the main theorem  in \cite{S2} (also see Theorem 4.5 in \cite{CL}).
\end{proof}

\begin{rema}
We can easily prove that any simple simple Harish-Chandra modules over the all above Lie superagebras  is a cuspidal module, or  a highest/lowest weight module as Lemma \ref{appN=1}. So all simple Harish-Chandra modules over the above Lie superalgebras are also classified.

\end{rema}

\begin{rema}
All indecomposible modules of the intermediate series and some other representations were studied in \cite{WGC} and \cite{WFL}.
\end{rema}

\noindent{\bf Acknowledgement:}  This work is partially supported by the NNSF (Nos. 12071405, 11971315, 11871249), and is partially supported by Xinjiang Uygur Autonomous Region graduate scientific research innovation project (No. XJ2021G021). The authors would like to
thank Prof. Rencai Lv for helpful discussions.


\begin{thebibliography}{00}

\bibitem{ADKP}E. Arbarello, C. De Concini,  V. Kac, and C. Procesi, Moduli spaces of curves and representation theory,
{ Commun. Math. Phys.} {\bf 117} (1988), 1-36.

\bibitem{Ba0}A. Bagchi, The BMS/GCA correspondence, Phys. Rev. Lett., {\bf 105} (2010), 171601.


\bibitem{BDMT}G. Barnich, L. Donnay, J. Matulich and R. Troncoso, {\it Asymptotic symmetries and dynamics
of three-dimensional flat supergravity}, JHEP 08 (2014), 071.


\bibitem{B} Y. Billig, {\it Jet modules,} Canad. J. Math., 59 (2007), no. 4, 721-729.

\bibitem{BF} Y. Billig, V. Futorny, {\it Classification of irreducible representations of Lie algebra of vector fields on a torus}, J. reine angew. Math.,  720 (2016): 199-216.

\bibitem{BFI} Y. Billig, V. Futorny, K. Iohara, {\it Classification of simple strong Harish-Chandra $W(m,n)$-modules}, arXiv: 2006.05618.


\bibitem{BBM} H. Bondi, M.G.J. van der Burg, A.W.K. Metzner, {\it Gravitational waves in general
relativity. 7. Waves from axisymmetric isolated systems,}
Proc. Roy. Soc. Lond. A, 269 (1962), 21.




\bibitem{CP} V. Chari, A. Pressley, {\it Unitary representations of the Virasoro algebra and a
conjecture of Kac}, Compositio Mathematica, 67 (1988), 315-342.


\bibitem{CLL}  Y. Cai, D. Liu, R. L\"{u}, {\it Classification of simple Harish-Chandra modules over the $N=1$ Ramond algebra}, J. Algebra, 567 (2021), 114-127.

\bibitem{CL} Y. Cai,   R. L\"{u}, {\it Classification of simple Harish-Chandra modules over the Neveu-Schwarz algebra and its contact subalgebra}, J. Pure Appl. Algebra, 226 (2022), 106866.


\bibitem{CLW} Y. Cai, R. L\"{u}, Y. Wang, {\it Classification of simple Harish-Chandra modules for map (super)algebras related to the Virasoro algebra}, J. Algebra, 570 (2021), 397-415.



\bibitem{DLM} P. Desrosiers, L. Lapointe,  P. Mathieu, {\it Superconformal field theory and Jack superpolynomials}, JHEP, { 09} (2012), 037.


\bibitem{Fuk}  D. B. Fuks, {\it Cohomology of Infinite-Dimensional Lie Algebras} [in Russian],  Nauka, Moscow (1984); English transl., Plenum, New York (1986).



\bibitem{FGM} V. Futorny, D. Grantcharov, V. Mazorchuk, {\it Weight modules over infinite dimensional Weyl algebras}, Proc. Amer. Math. Soc., 142 (2014), no. 9, 3049-3057.


\bibitem{G} Y. Ge, {\it Super Camassa–Holm-type systems associated to the Kuper–Ramond–Schwarz superalgebra}, Journal of Mathematical Physics, (2020), 61 (10): 103501.

\bibitem{GLZ}X. Guo, R. Lv, K. Zhao, {\it Simple Harish-Chandra modules,intermediate series modules, and Verma modules over the loop-Virasoro algebra}, Forum Math., 23 (2011), 1029-1052.



\bibitem{HLW} Y. He, D. Liu, Y. Wang, {\it Simple Harish-Chandra modules over the super affine-Virasoro algebras}, arxiv: 2112.07448.




%\bibitem{IK1}  K. Iohara, Y. Koga, {\it Representation theory of Neveu-Schwarz and Remond algebras I: Verma
%modules}. Adv. Math. 177(2003), 61-69.

%\bibitem{IK2} K. Iohara, Y. Koga,   {\it Representation theory of Neveu-Schwarz and Remond algebras II: Fock
%modules}. Ann. Inst. Fourier 53 (2003), 1755-1818.




\bibitem {Ka} V. Kac, {\it Some problems of infinite-dimensional Lie algebras and their representations},
Lecture Notes in Mathematics, 933 (1982), 117-126. Berlin, Heidelberg,
New York: Springer.


\bibitem {Ka1} V. Kac, {\it Superconformal algebras and transitive group actions on
quadrics}, Commun. Math. Phys., 186, (1997) 233-252.

\bibitem {KL} V. Kac,  J. van de Leuer, {\it On classification of superconformal
algebras}, Strings 88, Sinapore: World Scientific, (1988).




\bibitem{KS} I. Kaplansky, L. J. Santharoubane, {\it Harish-Chandra modules over the Virasoro algebra}, Infinite-dimensional groups with applications (Berkeley, Calif. 1984), 217--231, Math. Sci. Res. Inst. Publ., 4, Springer, New York, (1985).


\bibitem{L}D. Liu,  {Classification of Harish-Chandra
modules over some Lie algebras related to the Virasoro algebra}, { J. Algebra}  {
\bf 447} (2016), 548-559.


\bibitem{LPX1} D. Liu, Y. Pei, L. Xia, {\it Whittaker modules for the super-Virasoro algebras}, J. Algebra Appl.,  18 (2019),  1950211.

\bibitem{LPX2} D. Liu, Y. Pei, L. Xia, {\it Simple restricted modules for Ovsienko-Roger superalgebra}, J. Algebra, 546 (2020), 341-356.

\bibitem{LPX0} D. Liu, Y. Pei, L. Xia, {\it Classification of simple weight modules for the $N=2$ superconformal
algebra}, arXiv: 1904.08578v1.


\bibitem{LvZ} R. Lu, K. Zhao, {\it Classification of irreducible weight modules over the twisted
 Heisenberg-Virasoro algebra}, Commun. Contemp. Math., 12 (2010), no. 2, 183-205.




\bibitem{Ma} O. Mathieu,  {\it Classification of Harish-Chandra
modules over the Virasoro Lie algebra}, Invent. Math. 107 (1992), 225-234.



\bibitem{MZe} C. Mart\'{\i}nez, E. Zelmanov, {\it Graded modules over superconformal algebras}, Non-Associative and Non-Commutative Algebra and Operator Theory. Springer International Publishing, (2016), 41-53.


\bibitem{MZ} V. Mazorchuk, K. Zhao, {\it Supports of weight modules over Witt algebras. Proc. Roy. Soc. Edinburgh Sect. A}, 141 (2011), no. 1, 155-170.



\bibitem{OR} V.Yu. Ovsienko, C. Roger, {\it Extension of Virasoro group and
Virasoro algebra by modules of tensor densities on $S^1$},
Funct. Anal. Appl. {\bf 31} (1996), 4.

\bibitem{Rao} S. E. Rao, {\it Partial classification of modules for Lie algebra of diffeomorphisms of $d$-dimensional torus}, J. Math. Phys. 45 (2004), no. 8, 3322-3333.


\bibitem{S2} Y. Su. {\it Classification of Harish-Chandra modules over the super-Virasoro algebras},  Commun. Alg.,  23(10) (1995), 3653-3675.


\bibitem{S1}  Y. Su, {\it A classification of indecomposable $sl_2({\mathbb C})$-modules and a conjecture of Kac on irreducible modules over the Virasoro algebra}, J. Alg., 161(1993), 33-46.


\bibitem{WFL} H. Wang, H. Fa, J. Li, {\it Representations of Super W(2,2)  $\mathfrak L$}, arXiv: 1705.09452.

\bibitem{WCB} Y. Wang, Z. Chen,  C. Bai, {\it Classification of Balinsky-Novikov superalgebras with
dimension $2|2$}, J. Phys. A: Math. Theor., 45 (2012) 225201 (20pp).


\bibitem{WGC} Y. Wang, Q. Geng, Z. Chen, {\it The superalgebra of $W(2,2)$ and its modules of the intermediate series}, Commun. in Algebra, (2016), 45 (2), 749-763.





\bibitem{XL} Y. Xue, R. L\"{u}. {\it Simple weight modules with finite dimensional weight spaces over Witt superalgebras},  arXiv: 2001.04089.



\bibitem{ZD} W. Zhang, C. Dong, {$W$-algebra W(2,2) and the vertex operator algebra $L(\frac{1}{2},0)\otimes L(\frac{1}{2},0)$}, {Commun. Math. Phys.} {\bf 285} (2009), 991-1004.



\end{thebibliography}
\end{document}